\documentclass[12pt,reqno]{amsart}

\usepackage{amsmath}
\usepackage{amsfonts}
\usepackage{amsthm}
\usepackage{amstext}
\usepackage{amssymb}
\usepackage{txfonts}
\usepackage{enumerate}
\usepackage{graphics}
\usepackage{hyperref}
\usepackage{verbatim}

\newcommand{\R}{\mathbb{R}}

\newcommand{\eps}{\varepsilon}

\DeclareMathOperator*{\med}{median} 

\theoremstyle{plain}
\newtheorem{lemma}{Lemma} 
\newtheorem{theorem}{Theorem}

\theoremstyle{definition}
\newtheorem{definition}{Definition}

\begin{document}

\title[Asymptotics of $p$--harmonic functions]{Asymptotic statistical characterizations of $p$--harmonic functions of two variables}

\dedicatory{Dedicated to Professor Lloyd Jackson and his enduring legacy}

\author[D. Hartenstine]{David Hartenstine}
\email{david.hartenstine@wwu.edu}
\address{Department of Mathematics \\
Western Washington University \\
Bellingham, WA ~ 98225}

\author[M. Rudd]{Matthew Rudd}
\email{mbrudd@sewanee.edu}
\address{Department of Mathematics \\
Sewanee: The University of the South \\
Sewanee, TN ~ 37383}

\date{\today}
\subjclass[2010]{Primary: 35J92, 35J70, 35D40}
\keywords{$p$--Laplacian, mean-value property, $p$-harmonic functions, viscosity solutions}
\begin{abstract}
Generalizing the well-known mean-value property of harmonic functions,
we prove that a $p$--harmonic function of two variables satisfies, in a viscosity sense,
two asymptotic formulas involving its local statistics.  Moreover, we show that these asymptotic
formulas characterize $p$--harmonic functions when $1 < p < \infty$.  An example demonstrates
that, in general, these formulas do not hold in a non-asymptotic sense.
\end{abstract}

\maketitle

\begin{section}{Introduction}

A fundamental and fascinating fact about harmonic functions is their characterization by the
mean value property \cite{evans:pde98} : the continuous function $u$ is harmonic in the
domain $\Omega \subset \R^{N}$
if and only if
\begin{equation} \label{mean-vp}
u(x) = \fint_{\partial B_{r}(x)}{ u(s) \, ds } = \fint_{B_r (x)} {
u(y) \, dy } \quad \textrm{for each} \quad x \in \Omega \, ,
\end{equation}
where $B_{r}(x) \Subset \Omega$ is a ball with center $x$ and radius $r > 0$, $\partial B_{r}(x)$
is its boundary, and $\fint_{E}{ f }$ denotes the average of $f$ over the
set $E$.  Ostensibly, identity (\ref{mean-vp}) says nothing about derivatives and
could be studied entirely within the category of continuous functions.  It is the
prototypical \textit{statistical} characterization of
solutions of a PDE, and it is natural to wonder if this is peculiar to Laplace's equation.
In other words, can one characterize solutions of other PDEs in a statistical way that
avoids any explicit mention of derivatives?

Recent work shows that such statistical characterizations exist, in a certain sense,
for $p$--harmonic functions, i.e., solutions of the quasilinear PDE
\begin{equation} \label{p-laplace}
-\Delta_{p}u := - \operatorname{div}{ \left( |Du|^{p-2} Du \right) } = 0 \, , \quad \textrm{for} \quad
1 < p < \infty .
\end{equation}
More precisely, $p$--harmonic functions are usually defined to be
weak solutions of (\ref{p-laplace}); thanks to work by Juutinen et
al. \cite{juutinen:evs01}, however, weak solutions of
(\ref{p-laplace}) are the same as viscosity solutions of
(\ref{p-laplace}).  Viscosity techniques are particularly relevant
to the present work, as Manfredi et al. \cite{manfredi:amv10} used
such methods to prove that the continuous function $u$ is
$p$--harmonic in the domain $\Omega \subset \R^{N}$ if and only if
the functional equation
\begin{equation} \label{manfredi}
u(x) = \frac{\alpha}{2} \left\{ \max_{ \, \overline{B_{\eps}(x) \, } }{ u } +
\min_{ \, \overline{B_{\eps}(x)} \, }{ u } \right\} +
\beta \fint_{B_{\eps}(x)}{ u(y) \, dy } + o( \eps^{2} ) \quad \textrm{as} \quad \eps \to 0
\end{equation}
holds in the viscosity sense for all $x \in \Omega$. The constants $\alpha$ and $\beta$ are
determined by the exponent $p$ and the dimension $N$:
\[
\alpha := \frac{p-2}{p+N} \quad \textrm{and} \quad \beta := \frac{2+N}{p+N} \, .
\]
This characterization also holds for $\infty$--harmonic functions, where
the $\infty$--Laplacian $\Delta_{\infty}$ has the formal definition
\begin{equation} \label{infty-lap}
\Delta_{\infty} u :=
\frac{1}{ |Du|^{2} } \sum_{i,j=1}^{N}{ \frac{\partial u}{\partial x_{i}} \, \frac{\partial u}{\partial x_{j}} \, \frac{\partial^{2} u}{ \partial x_{i} \partial x_{j} } }
\end{equation}
for smooth $u$.

To establish their results, the authors of \cite{manfredi:amv10} combine several interesting
facts.  First, calculating formally yields
\begin{equation} \label{decomp-1}
\Delta_{p} u = |Du|^{p-2} \left( \Delta u + (p-2) \Delta_{\infty} \right) \, ,
\end{equation}
an identity that plays a central role in both \cite{juutinen:evs01} and \cite{manfredi:amv10}.
Using it, Juutinen et al. proved that $u$ is a viscosity solution of (\ref{p-laplace}) if and
only if
\[
-\Delta u - (p-2) \Delta_{\infty} u = 0
\]
in the viscosity sense, about which more will be said below.
Manfredi et al. then invoke the identities
\begin{equation} \label{taylor-2-lap}
u(x) - \fint_{B_{\eps}(x)}{ u(y) \, dy } = -\frac{ \eps^{2} }{ \,
2(N+2) \, } \Delta u(x)+ o( \eps^{2} )
\end{equation}
and
\begin{equation} \label{taylor-infty-lap}
u(x) - \frac{1}{2} \left\{ \max_{ \, y \in \overline{B_{\eps}(x)
\, } }{ u(y) } + \min_{ \, y \in \overline{B_{\eps}(x) \, } }{
u(y) } \right\} = -\frac{ \eps^{2} }{ 2 } \Delta_{\infty} u(x) +
o( \eps^{2} ) \, ,
\end{equation}
valid for smooth functions as $\eps \rightarrow 0$, to obtain
their asymptotic characterization (\ref{manfredi}).  Here and in
what follows, a function is called smooth if it is $C^2$.

The decomposition (\ref{decomp-1}) can be written in various ways,
a fact that we exploit to obtain new statistical characterizations
of $p$--harmonic functions of two variables.  Specifically, if we
define the $1$--Laplacian $\Delta_{1}$ on smooth functions by
\begin{equation} \label{1-harmonic}
\Delta_{1} u :=
| Du | \operatorname{div}{ \left( \frac{ Du }{ \, |Du| \, } \right) } \, ,
\end{equation}
then the formal relationship
\[
\Delta_{1} = \Delta - \Delta_{\infty}
\]
holds and immediately yields two alternatives to (\ref{decomp-1}) :
\begin{equation} \label{decomp-2}
\Delta_{p} u = |Du|^{p-2} \left( \, (p-1) \Delta u + (2-p) \Delta_{1} u \, \right) \, ,
\end{equation}
and
\begin{equation} \label{decomp-3}
\Delta_{p} u = |Du|^{p-2} \left( \, \Delta_{1} u + (p-1) \Delta_{\infty} u \, \right) \, .
\end{equation}
Using these identities and the Taylor approximation
\begin{equation} \label{taylor-1-lap}
u(x) - \med_{s \in \partial B_{\eps}(x)}{ \left\{ u(s) \right\} }
= -\frac{ \eps^{2} }{ 2 } \Delta_{1} u (x) + o( \eps^{2} ) \, ,
\end{equation}
valid for smooth functions $u$ of two variables as $\eps
\rightarrow 0$, we prove the following:
\begin{theorem} \label{main-thm}
Suppose that $1 < p < \infty$ and $\Omega \subset \R^{2}$ is open,
and let $u$ be a continuous function on $\Omega$.  The following
are equivalent:
\begin{enumerate}
\item
$u$ is $p$--harmonic in $\Omega$.
\item
At each $x \in \Omega$, the equation
\begin{equation} \label{fe-1}
u(x) = \left( \frac{2}{p} - 1 \right) \med_{s \in \partial
B_{\eps}(x) }{ \left\{ \, u(s) \, \right\} } + \left( 2 -
\frac{2}{p} \right) \fint_{\partial B_{\eps}(x)}{ u(s) \, ds } +
o( \eps^{2} )\quad \textrm{as} \quad \eps \to 0
\end{equation}
holds in the viscosity sense.
\item
At each $x \in \Omega$, the equation
\begin{equation} \label{fe-2}
u(x) = \frac{1}{p} \med_{s \in \partial B_{\eps}(x)}{ \left\{ \,
u(s) \, \right\} } + \left( \frac{p-1}{2p} \right) \left( \max_{
\, y \in \overline{B_{\eps}(x) \, } }{ \left\{ \, u(y) \, \right\}
} + \min_{ \, y \in \overline{B_{\eps}(x) \, } }{ \left\{ \, u(y)
\, \right\} } \right) + o( \eps^{2} )\quad \textrm{as} \quad \eps
\to 0
\end{equation}
holds in the viscosity sense.
\end{enumerate}
\end{theorem}
\noindent The median operator occurring here is defined as expected: if $u$ is continuous on
$\Omega$, $x \in \Omega$, and $\overline{ B_{\eps}(x) } \subset \Omega$,
\[
m = \med_{ \, s \in \partial B_{\eps}(x) \, }{ \left\{ \, u(s) \, \right\} }
\]
if and only if
\[
| \, \left\{ s \in \partial B_{\eps}(x) \, : \, u(s) \geq m \, \right\} \, | =
| \, \left\{ s \in \partial B_{\eps}(x) \, : \, u(s) \leq m \, \right\} \, | \, ,
\]
where $|E|$ is the 1-dimensional Hausdorff measure of the set $E$.
We remark that if $u$ is smooth and $|Du(x)| \ne 0$, then
(\ref{fe-1}) and (\ref{fe-2}) hold in the usual non-viscosity
sense if and only if $\Delta_p u(x) = 0$.  This follows from
Lemmas 1 and 2 below.

Considering (\ref{mean-vp}), it is natural to ask if the formulas (\ref{fe-1}) and (\ref{fe-2}) hold in a non-asymptotic sense.  More precisely, if $u$ is $p$--harmonic in $\Omega$, do the equations
\begin{equation}\label{nonasympfe-1}
u(x) = \left( \frac{2}{p} - 1 \right) \med_{s \in \partial B_{\eps}(x) }{ \left\{ \, u(s) \, \right\} } +
\left( 2 - \frac{2}{p} \right) \fint_{\partial B_{\eps}(x)}{ u(s) \, ds }
\end{equation}

\begin{equation} \label{nonasympfe-2}
u(x) = \frac{1}{p} \med_{s \in \partial B_{\eps}(x)}{ \left\{ \, u(s) \, \right\} } +
\left( \frac{p-1}{2p} \right)
\left( \max_{ \, y \in \overline{B_{\eps}(x) \, } }{ \left\{ \, u(y) \, \right\} } +
\min_{ \, y \in \overline{B_{\eps}(x) \, } }{ \left\{ \, u(y) \, \right\} } \right)
\end{equation}
\noindent
necessarily hold at all $x \in \Omega$ for all $\eps>0$ sufficiently small?
The answer to this question is no, and in
Section~\ref{counterexamples} we provide an example demonstrating
that these equations do not hold in general even for smooth $p$--harmonic functions.

On the way to proving Theorem \ref{main-thm} in Section \ref{proof}, we provide a simple
analytic proof of identity (\ref{taylor-1-lap}).
We should point out, however, that the relationship between median values and the
$1$--Laplacian has appeared before, either explicitly or
implicitly.  In \cite{oberman:cmd04}, for example, Oberman uses a
discrete median scheme of forward Euler type to approximate solutions of the parabolic mean curvature
equation,
\begin{equation} \label{parabolic}
\frac{ \partial u }{ \partial t } - \Delta_{1} u = 0 \quad \textrm{for} \quad t > 0 \, , \quad
u(\cdot,0) = u_{0} \, ,
\end{equation}
in two space dimensions.  Unlike many other proposed algorithms for this equation, Oberman's
median scheme is provably convergent, an easy consequence
of the main theorem in \cite{barles:cas91}.

Kohn and Serfaty \cite{kohn:dcb06} discuss a
different convergent approximation scheme for
the initial--value problem (\ref{parabolic}) that can be described geometrically as follows.
Let $\Gamma(0)$ be a simple closed curve in the plane, let $\Gamma(t)$ be the curve obtained from
$\Gamma(0)$ by letting it evolve by mean curvature for time $t$, and fix a small $\eps > 0$.
The curve $\Gamma( t + \frac{\eps^{2}}{2} )$ is approximately the locus of all centers of circles of
radius $\eps$ with antipodal points on $\Gamma(t)$; one can approximate $\Gamma( t + \frac{\eps^{2}}{2} )$
by tracking the center of a segment of length $2\eps$ as its endpoints traverse the curve $\Gamma(t)$.
This is the basic idea behind our proof of (\ref{taylor-1-lap}), even though Kohn and Serfaty
never mention medians in \cite{kohn:dcb06}.  Related papers
that use similar ideas without explicitly connecting the $1$--Laplacian and median values include,
but are certainly not limited to, \cite{catte:msm95} and  \cite{ruuth:cgm00}.

The present work is actually closely related to the work of
Jackson and it is our pleasure to briefly discuss this connection.
Over the past thirty or so years, viscosity solutions have become
a standard tool in the study of nonlinear PDEs.  However the contemporary viscosity approach is similar
in some ways to the earlier abstract Perron method of Jackson and
Jackson and Beckenbach as in \cite{beckenbach:ssv53},
\cite{jackson:gsf55} and \cite{jackson:sdp58}. In fact, for a
class of second-order elliptic PDEs,
viscosity subsolutions and the subfunctions of Beckenbach and
Jackson are equivalent (see \cite{hartenstine:gvs04}). Furthermore,
Jackson applied this abstract Perron method to obtain existence and
uniqueness results for the minimal surface equation in two
independent variables \cite{jackson:sdp58}; this work is closely related
to ongoing work on $1$--harmonic functions \cite{rudd:mvo}, as the level
sets of $1$--harmonic functions are minimal surfaces (cf. \cite{ziemer:flg99}).

\end{section}

\begin{section}{New results} \label{results}

\begin{subsection}{Definitions}

Before proving Theorem \ref{main-thm}, we review the necessary definitions and related results.

\begin{definition} \label{sub-super-soln}
Suppose that $1 < p < \infty$, and let $\Omega$ be a domain in $\R^{2}$.
\begin{enumerate}
\item The lower semicontinuous function $u$ is {\it
$p$-superharmonic} in $\Omega$ in the viscosity sense if and only
if the equivalent inequalities
\begin{equation} \label{p-superharmonic}
(1-p) \Delta \varphi + (p-2) \Delta_{1} \varphi \geq 0 \quad \textrm{and} \quad
-\Delta_{1} \varphi + (1-p) \Delta_{\infty} \varphi \geq 0
\end{equation}
hold at $x \in \Omega$ for any smooth function $\varphi$ such that
$|D\varphi(x)| \neq 0$ and $u - \varphi$ has a strict minimum at
$x$.

\item The upper semicontinuous function $u$ is {\it
$p$-subharmonic} in $\Omega$ in the viscosity sense if and only if
the equivalent inequalities
\begin{equation} \label{p-subharmonic}
(1-p) \Delta \varphi + (p-2) \Delta_{1} \varphi \leq 0 \quad \textrm{and} \quad
-\Delta_{1} \varphi + (1-p) \Delta_{\infty} \varphi \leq 0
\end{equation}
hold at $x \in \Omega$ for any smooth function $\varphi$ such that
$|D\varphi(x)| \neq 0$ and $u - \varphi$ has a strict maximum at
$x$.

\item $u$ is {\it $p$-harmonic} in $\Omega$ if it is both
$p$-superharmonic and $p$-subharmonic in $\Omega$.

\end{enumerate}

\end{definition}

The legitimacy of this definition follows from \cite{juutinen:evs01} and the formal identities
(\ref{decomp-1}), (\ref{decomp-2}) and (\ref{decomp-3}) above,
as checking $p$--harmonicity in the viscosity sense reduces to evaluating $-\Delta_{p} \varphi$
for smooth functions $\varphi$ away from critical points.  We refer to \cite{juutinen:evs01}
and \cite{manfredi:amv10} for more details.

\begin{definition} \label{fe-sub-super}

Let $1 < p < \infty$, let $\Omega$ be a domain in $\R^{2}$, and consider the equation
\begin{equation} \label{fe}
u(x) = \left( \frac{2}{p} - 1 \right) \med_{s \in \partial
B_{\eps}(x) }{ \left\{ \, u(s) \, \right\} } + \left( 2 -
\frac{2}{p} \right) \fint_{\partial B_{\eps}(x)}{ u(s) \, ds } +
o( \eps^{2} ) \, \quad \textrm{as} \quad \eps \to 0.
\end{equation}

\begin{enumerate}

\item $u$ is a {\it supersolution of (\ref{fe}) in the viscosity
sense} if and only if the inequality
\begin{equation} \label{fe-super}
\varphi(x) \geq \left( \frac{2}{p} - 1 \right) \med_{s \in
\partial B_{\eps}(x) }{ \left\{ \, \varphi(s) \, \right\} } +
\left( 2 - \frac{2}{p} \right) \fint_{\partial B_{\eps}(x)}{
\varphi(s) \, ds } + o( \eps^{2} )\quad \textrm{as} \quad \eps \to
0
\end{equation}
holds at $x \in \Omega$ for any smooth function $\varphi$ such
that $|D\varphi(x)| \neq 0$ and $u - \varphi$ has a strict minimum
at $x$.

\item $u$ is a {\it subsolution of (\ref{fe}) in the viscosity
sense} if and only if the inequality
\begin{equation} \label{fe-sub}
\varphi(x) \leq \left( \frac{2}{p} - 1 \right) \med_{s \in
\partial B_{\eps}(x) }{ \left\{ \, \varphi(s) \, \right\} } +
\left( 2 - \frac{2}{p} \right) \fint_{\partial B_{\eps}(x)}{
\varphi(s) \, ds } + o( \eps^{2} )\quad \textrm{as} \quad \eps \to
0
\end{equation}
holds at $x \in \Omega$ for any smooth function $\varphi$ such
that $|D\varphi(x)| \neq 0$ and $u - \varphi$ has a strict maximum
at $x$.

\item $u$ is a {\it solution of (\ref{fe}) in the viscosity sense}
if and only if it is both a subsolution and a supersolution.

\end{enumerate}
\end{definition}

\end{subsection}

\begin{subsection}{Proof of Theorem \ref{main-thm}} \label{proof}

We begin with asymptotic formulas valid for smooth functions that
will be used to establish our main result.  The following lemma
can be established using Taylor expansion; we omit the routine
proof.

\begin{lemma}
Let $\Omega$ be a domain in $\R^{2}$, let $x \in \Omega$, and let
$\varphi$ be a smooth function on $\Omega$.
Then
\begin{equation}
\varphi(x) - \fint_{\partial B_{\eps}(x)}{ \varphi(s) \, ds } = -
\frac{\eps^{2}}{4} \Delta \varphi (x) + o( \eps^{2} ) \, \quad
\textrm{as} \quad \eps \to 0.
\end{equation}
\end{lemma}


\begin{lemma}
Let $\Omega$ be a domain in $\R^{2}$, let $x = (x_{1},x_{2}) \in
\Omega$, and let $\varphi$ be a smooth function on $\Omega$ with
$|D \varphi(x)| \neq 0$. Then
\begin{equation} \label{taylor-1-lap-copy}
\varphi(x) - \med_{s \in \partial B_{\eps}(x)}{ \left\{ \varphi(s)
\right\} } = -\frac{ \eps^{2} }{ 2 } \Delta_{1} \varphi(x) + o(
\eps^{2} ) \, \quad \textrm{as} \quad \eps \to 0.
\end{equation}
\end{lemma}

\begin{proof}

The Implicit Function Theorem guarantees that, for $\eps > 0$ sufficiently small, the level
sets of $\varphi$ form a one--parameter family of smooth, non--intersecting curves that foliate
the closed ball $\overline{ B_{\eps}(x) } $.  Consequently, the median of $\varphi$ over $\partial B_{\eps}(x)$,
\[
M_{\eps} := \med_{ \, s \in \partial B_{\eps}(x) \, }{ \left\{ \varphi(s) \right\} } \, ,
\]
is the value corresponding to the level set that intersects $\partial B_{\eps}(x)$ in antipodal points;
for each $\eps > 0$, there is a unique
angle $\theta_{\eps} \in [0, 2\pi)$ such that
\begin{equation} \label{antipode}
M_{\eps} =
\varphi( x_1 + \eps \cos{ \theta_{\eps} }, x_2 + \eps \sin{ \theta_{\eps} } ) =
\varphi( x_1 - \eps \cos{ \theta_{\eps} }, x_2 - \eps \sin{ \theta_{\eps} } )  \, .
\end{equation}

Let $\mathbf{v}_{\eps}$ denote the unit vector $(\cos{ \theta_{\eps}}, \sin{ \theta_{\eps}})$,
and define
\[
D\varphi^{\perp}(x) :=  ( - \varphi_{2}(x), \varphi_{1}(x) ) \, .
\]
The derivatives of $\varphi$ below are evaluated
at $x$, which we omit for simplicity.
Taylor expanding about $x$ yields
\begin{equation} \label{taylor1}
M_{\eps} = \varphi( x + \eps \mathbf{v}_{\eps} ) =
\varphi( x ) + \eps D\varphi \cdot \mathbf{v}_{\eps} +
\frac{ \eps^{2} }{2} \mathbf{v}_{\eps}^{\intercal} D^{2}\varphi \mathbf{v}_{\eps} + o( \eps^{2} )
\end{equation}
and
\begin{equation} \label{taylor2}
M_{\eps} = \varphi( x - \eps \mathbf{v}_{\eps} ) =
\varphi( x ) - \eps D\varphi \cdot \mathbf{v}_{\eps} +
\frac{ \eps^{2} }{2} \mathbf{v}_{\eps}^{\intercal} D^{2}\varphi \mathbf{v}_{\eps} + o( \eps^{2} ) \, .
\end{equation}
Since these expressions both equal $M_{\eps}$,
\[
\eps D\varphi \cdot \mathbf{v}_{\eps} = o( \eps^{2} ) \, .
\]
We therefore have
\begin{equation} \label{decomp}
\mathbf{v}_{\eps} = \frac{ D\varphi^{\perp} }{ \, | D \varphi | \, } + \mathbf{w}_{\eps} \, ,
\end{equation}
where
\[
\eps D\varphi \cdot \mathbf{w}_{\eps} = o( \eps^{2} ) \, ,
\]
and we see (among other things) that the sequence $\{
\mathbf{v}_{\eps} \}$ of unit vectors converges:
\[
\mathbf{v}_{\eps} \to \frac{ D \varphi^{\perp} }{ \, | D \varphi |
\, } \quad \textrm{as} \quad \eps \downarrow 0 \, .
\]
Using the decomposition (\ref{decomp}) in the right--hand side of either (\ref{taylor1}) or (\ref{taylor2})
yields (cf. \cite{kohn:dcb06})
\begin{equation}
\varphi(x) - M_{\eps} = - \frac{ \eps^{2} }{2} \frac{(
D\varphi^{\perp})^{\intercal}}{|D \varphi|} \, D^{2}\varphi \,
\frac{D\varphi^{\perp}}{| D \varphi |} + o( \eps^{2} ) = -\frac{
\eps^{2} }{2} \Delta_{1} \varphi + o( \eps^{2} ) \, ,
\end{equation}
proving the lemma.

\end{proof}

With these lemmas, Theorem~\ref{main-thm} is easily established using the same approach as in \cite{manfredi:amv10}: apply the asymptotic formulas for smooth functions to the viscosity formulation.

\begin{proof}

Suppose that $u$ is continuous in $\Omega$ and that $\varphi$ is a
smooth function for which $|D \varphi(x)| \ne 0$ and $u- \varphi$
has a strict minimum at $x \in \Omega.$  Using Lemmas 1 and 2 and
observing that $(2/p -1) + (2-2/p) = 1$, it follows that the first
inequality in (\ref{p-superharmonic}) holds if and only if
(\ref{fe-super}) holds.  Thus $u$ is $p$--superharmonic in the
viscosity sense if and only if it is a viscosity supersolution of
(\ref{fe-1}).  The analogous argument establishes the equivalence
of $p$--subharmonicity and being a subsolution of (\ref{fe-1}).

The equivalence of the first and third statements of the theorem
is proved similarly, using identity (\ref{taylor-infty-lap})
instead of Lemma 1.

\end{proof}

\end{subsection}

\begin{subsection}{Necessity of Asymptotic Nature of Theorem~\ref{main-thm}} \label{counterexamples}

In this section, we present an example to show that (\ref{nonasympfe-1}) and (\ref{nonasympfe-2}) do not hold for $p$--harmonic functions in general.  In fact, these equations do not even necessarily hold for all $\eps>0$ sufficiently small, so that the asymptotic results appearing in Theorem~\ref{main-thm} are, in general, the best available.

For any $1 < p < 2$, the function $u_p (x) = |x|^{(p-2)/(p-1)}$ is
smooth and $p$-harmonic in $\mathbb{R}^2 \setminus \{0\}$, and is
known as the fundamental solution of the $p$-Laplacian (see for example \cite{juutinen:evs01}).  Let $x=(x_1,0)$ where $x_1 >0$ and let $0< \eps < x_1$.  Because $u_p$ is radial and radially decreasing, it is not hard to see that
\begin{equation} \label{fundsolmed}
\med_{\partial B_{\eps}(x)} u_p = (x_1^2 + \eps^2)^{(p-2)/2(p-1)}.
\end{equation}
The mean of $u_p$ on $\partial B_{\eps}(x)$ is
\begin{equation}\label{fundsolmean}
\frac{1}{2 \pi} \int_0^{2 \pi} (x_1^2 + 2 x_1 \eps \cos \theta +
\eps^2)^{(p-2)/2(p-1)} \, d \theta.
\end{equation}
Using (\ref{fundsolmed}) and (\ref{fundsolmean}), (\ref{nonasympfe-1}) at $x$ with $u=u_p$ becomes
\begin{equation}\label{fundsolform-1}
|x_1|^{(p-2)/(p-1)} = \left(\frac{2}{p}-1\right) (x_1^2 + \eps^2)^{(p-2)/2(p-1)} + \left(2 - \frac{2}{p}\right) \frac{1}{2 \pi} \int_0^{2 \pi} (x_1^2 + 2 x_1 \eps \cos
\theta + \eps^2)^{(p-2)/2(p-1)} \, d \theta.
\end{equation}
If (\ref{fundsolform-1}) holds for all $\eps$ sufficiently small we can differentiate it with respect to $\eps$ to obtain
\begin{equation}\label{fundsolform-2}
(2-p) (x_1^2 + \eps^2)^{(p-2)/2(p-1)-1} \eps = \frac{2-2p}{2
\pi}\int_0^{2 \pi} (x_1^2 + 2 x_1 \eps \cos \theta +
\eps^2)^{(p-2)/2(p-1)-1} (x_1 \cos \theta + \eps) \, d \theta.
\end{equation}
Now let $x_1 = 1$ and $p=3/2$.  The last equation is then
\begin{equation}
(1/2)(1+\eps^2)^{-3/2} \eps = \frac{-1}{2 \pi} \int_0^{2
\pi} (1 + 2 \eps \cos \theta + \eps^2)^{-3/2} ( \cos
\theta + \eps) \, d \theta,
\end{equation}
which holds if and only if
\begin{equation}\label{epsintform}
- \eps = \frac{1}{\pi} \int_0^{2 \pi} \left( \frac{1 + 2
\eps \cos \theta + \eps^2}{1 + \eps^2}\right)^{-3/2} (
\cos \theta + \eps) \, d \theta = \frac{1}{\pi} \int_0^{2 \pi} \left(1 + \frac{2 \eps \cos
\theta}{1 + \eps^2} \right)^{-3/2} (
\cos \theta + \eps) \, d \theta.
\end{equation}
Using the binomial formula:
\begin{equation}
\left(1 + \frac{2 \eps \cos
\theta}{1 + \eps^2} \right)^{-3/2} =
1 - \frac{3}{2}
\left(\frac{2 \eps \cos \theta}{1 + \eps^2} \right) +
\frac{15}{8}\left(\frac{2 \eps \cos \theta}{1 + \eps^2}
\right)^2 - \frac{35}{16} \left(\frac{2 \eps \cos \theta}{1 +
\eps^2} \right)^3
\end{equation}
plus higher order terms.  Therefore the integrand in (\ref{epsintform}) is equal to
\begin{equation} \label{asympint}
\cos \theta
- \frac{3 \eps}{1 + \eps^2} \cos^2 \theta + \frac{15}{2}
\frac{\eps^2 \cos^3 \theta}{(1 + \eps^2)^2}-\frac{35}{2} \frac{\eps^3 \cos^4 \theta}{(1 + \eps^2)^3}
+ \eps - \frac{3 \eps^2 \cos \theta}{(1 + \eps^2)} +
\frac{15}{2}\frac{\eps^3 \cos^2 \theta}{(1 + \eps^2)^2}
\end{equation}
plus terms of order 4 and higher.  Using (\ref{asympint}) in the integral in (\ref{epsintform}), noting that odd powers of $\cos \theta$ integrate to zero and recalling that $\int_0^{2
\pi} \cos^2 \theta \, d \theta = \pi$ and $\int_0^{2 \pi} \cos^4
\theta \, d \theta = (3/4)\pi$, we obtain
\begin{equation}
\frac{1}{\pi} \int_0^{2 \pi} \left( \frac{1 + 2 \eps \cos
\theta + \eps^2}{1 + \eps^2}\right)^{-3/2} ( \cos \theta +
\eps) \, d \theta \approx -\eps - (21/8) \eps^3,
\end{equation}
which is strictly less than $-\eps$ if $\eps$ is sufficiently small so that (\ref{epsintform}) does not hold.  As a result, (\ref{fundsolform-1}) cannot hold for all $\eps$ sufficiently small.

The same example can be used to show that (\ref{nonasympfe-2}) also fails in general, even if $\eps$ is small.  Again let $p=3/2$ and $x = (1,0)$, and let $0< \eps < 1.$  The maximum value of $u_p$ on $\overline{B_\eps (x)\,}$ is $1/(1-\eps)$ and the minimum on the same ball is $1/(1+\eps)$.  Using (\ref{fundsolmed}), in this case (\ref{nonasympfe-2}) becomes
\begin{equation}
1 = \frac{2}{3}\left(1 + \eps^2 \right)^{-1/2} + \frac{1}{6}\left(\frac{1}{1-\eps} + \frac{1}{1 + \eps} \right)
\end{equation}
which one can easily see does not hold, even if $\eps>0$ is restricted to being smaller than some $\eps_0$.

\end{subsection}

\end{section}

\begin{section}{Concluding remarks}

The asymptotic characterizations of $p$--harmonic functions in \cite{manfredi:amv10} are
valid in $N$ dimensions.  It would be interesting to extend the results presented here to
higher dimensions.  The only part of the proof of Theorem 1 that requires two dimensions
is Lemma 2.  If an $N$-dimensional version of Lemma 2, perhaps involving the median
on an $(N-1)$--dimensional sphere, were established, new asymptotic statistical
characterizations of $p$-harmonic functions would follow.

We presented an example showing that, in general, only asymptotic characterizations of this type are possible.  However, this is not the case for $p=2$.  A natural question is: do the
equations (\ref{nonasympfe-1}) and (\ref{nonasympfe-2}) hold either
globally or locally for any other values of $p$?
Concrete examples in \cite{rudd:mvo} show that the limiting cases of (\ref{nonasympfe-1})
and (\ref{nonasympfe-2}) can hold when $p=1$,
but more work on this question needs to be done.

Finally, we did not consider the extreme cases $p=1$ and $p=\infty$, although we remark that if $p$ is formally allowed to be $\infty$ in (\ref{fe-2}) the resulting characterization is the same as that in \cite{manfredi:amv10}.

\end{section}

\bigskip
\bigskip

\end{document}